\documentclass{amsart}

\usepackage{amsmath}
\usepackage{amssymb}
\usepackage{mathrsfs}
\usepackage{amsthm}
\usepackage{enumerate}

\newtheorem{thm}{Theorem}[section]

\newtheorem{lem}[thm]{Lemma}

\theoremstyle{definition}

\newtheorem{conjecture}[thm]{Conjecture}

\numberwithin{equation}{section}

\title[Weak type estimates]{Weak type commutator and Lipschitz estimates: resolution of the Nazarov-Peller conjecture}

\author{M. Caspers, D. Potapov, F. Sukochev, D. Zanin}
\date{\today, {\it MSC2000}: 47B10, 47L20, 47A30,\\ {\it Acknowledgement:} The first author is supported by the grant SFB 878. The work of other co-authors is supported by the ARC}

\address{M. Caspers, Fachbereich Mathematik und Informatik der Universit\"at M\"unster,
Einsteinstrasse 62,
48149 M\"unster, Germany}
\email{martijn.caspers@uni-muenster.de}
\address{D. Potapov, F. Sukochev, D. Zanin, School of Mathematics and Statistics, UNSW, Kensington 2052, NSW, Australia}
\email{d.potapov@unsw.edu.au}
 \email{f.sukochev@unsw.edu.au}
\email{d.zanin@unsw.edu.au}

\begin{document}

\begin{abstract}
Let $\mathcal{M}$ be a semi-finite von Neumann algebra and let $f: \mathbb{R} \rightarrow \mathbb{C}$ be a Lipschitz function. If $A,B\in\mathcal{M}$ are self-adjoint operators such that $[A,B]\in L_1(\mathcal{M}),$ then
$$\|[f(A),B]\|_{1,\infty}\leq c_{abs}\|f'\|_{\infty}\|[A,B]\|_1,$$
where $c_{abs}$ is an absolute constant independent of $f$, $\mathcal{M}$ and $A,B$ and $\|\cdot\|_{1,\infty}$ denotes the weak $L_1$-norm.
If $X,Y\in\mathcal{M}$ are self-adjoint operators such that $X-Y\in L_1(\mathcal{M}),$ then
$$\|f(X)-f(Y)\|_{1,\infty}\leq c_{abs}\|f'\|_{\infty}\|X-Y\|_1.$$
This result resolves a conjecture raised by F. Nazarov and V. Peller implying a couple of existing results in perturbation theory.
\end{abstract}

\maketitle

\section{Introduction}
Let $L_p(H)$ be the Schatten-von Neumann ideal of $B(H)$. It consists of all compact operators for which its sequence of singular values lies in $\ell_p.$ Let $F_p$ be the class of functions $f: \mathbb{R} \rightarrow \mathbb{C}$ such that
$$f(B) - f(C) \in L_p(H),$$
for all self-adjoint $B,C$ such that $B-C\in L_p(H)$ and set
$$\Vert f \Vert_{F_p} = \sup_{B\neq C} \frac{\Vert f(B) - f(C)\Vert_p}{\Vert B - C \Vert_p}.$$
It was conjectured by M.G. Krein \cite{Krein} that whenever the  derivative $f' \in L_\infty(\mathbb{R})$ we have $f \in F_1$. This conjecture does not hold as was shown by Y.B. Farforovskaya in \cite{Far8}. Also it was shown that the analogue of Krein's problem fails in the case $p = \infty$ (see \cite{Far6, Far7}). In fact already for the absolute value function it was found by T. Kato that Krein's problem has a negative answer \cite{Kato};  and similarly in the case $p=1$ by E.B. Davies \cite{Davies}.

A positive result in this direction was first obtained by M. Birman and M. Solomyak \cite[Theorem 10]{BiSo1} who proved that $C^{1+\epsilon} \subseteq F_1$ for every $\epsilon>0$, and later improved by V. Peller \cite{PellerHankel} who showed that $B_{\infty 1}^1 \subseteq F_1$. Here $B_{pq}^s$ is the class of Besov spaces for which we refer to \cite{Grafakos}.  The Krein problem for the case $1< p < \infty, p \not = 2$ remained open until \cite{PotapovSukochev2}. In \cite{PotapovSukochev2} it was shown by the second and third named author that $F_p$ consists exactly of all Lipschitz functions. Moreover in \cite{CMPS} a quantitative estimate for $\Vert f \Vert_{F_p}$ was found, namely $\Vert f \Vert_{F_p} \simeq p^2/(p-1)$. Earlier the same problem had been considered by M. de la Salle (unpublished, see \cite{Salle}) who was able to show already that  $\Vert f \Vert_{F_p} \leq c_{abs}\cdot p^4/(p-1)^2$ with an absolute constant $c_{abs},$ which is more optimal than \cite{PotapovSukochev2}. 

Other results concerning this problem  have been obtained in \cite{DDPS2} and \cite{Kosaki} and in the context of this paper we also mention \cite{Ran} in which weak estimates for martingale inequalities were obtained.

\vspace{0.3cm}

Using interpolation, the above results would follow from a weak type Lipschitz estimate between $L_1$ and the weak-$L_1$ space $L_{1,\infty}$. The estimate was conjectured in a paper of F. Nazarov and V. Peller \cite{NazarovPeller}, and has remained the major open question in the study of Lipschitz properties of operator valued functions.
Denote $L_{1, \infty}(H)$ for the weak $L_1$-space consisting of all compact operators $A$ whose sequence $\{\mu(k,A)\}_{k\geq0}$ of singular values satisfies $\mu(k,A)=O(\frac1{k+1}).$

\begin{conjecture} \label{Conjecture}
 Let $f: \mathbb{R} \rightarrow \mathbb{C}$ be Lipschitz. Whenever $A, B \in B(H)$ are self-adjoint operators such that $A - B \in L_1(H)$, we have that $f(A) - f(B)  \in L_{1, \infty}(H)$ and
$$\Vert f(A) - f(B) \Vert_{1, \infty} \leq c_{abs}\|f'\|_\infty \Vert A - B\Vert_1,$$
 for some absolute constant $c_{abs}$.
\end{conjecture}

Nazarov and Peller \cite{NazarovPeller} gave an affirmative answer under the assumption that the rank of $A-B$ equals 1. Since $\Vert \: \cdot \: \Vert_{1, \infty}$ is a quasi-norm and not a norm for $L_{1, \infty}(H)$ it is impossible to extend their result for when $A-B$ is a general trace class operator.

Another positive result to the conjecture was found by the current authors in \cite{CPSZ} in the special case when $f$ is the absolute value map. The proof relies on the observation that the Schur multiplier of divided differences
\[
\left( \frac{f(\lambda) - f(\mu)}{ \lambda - \mu} \right)_{\lambda \not = \mu}
\]
can be written as a finite sum of compositions of a positive definite Schur multiplier and a triangular truncation operator. For general Lipschitz functions there is no reason that the latter fact should be true which renders the technique of \cite{CPSZ} inapplicable. 

\vspace{0.3cm}

The main result of this paper is a proof of Conjecture \ref{Conjecture}. The importance of this result lies in the fact that this gives the sharpest possible estimate for perturbations and commutators. 
In particular it retrieves $\Vert f \Vert_{F_p} \simeq p^2/(p-1)$ \cite{CMPS} and the Nazarov--Peller result \cite{NazarovPeller}. A key ingredient in our proof is the connection with non-commutative Calder\'on-Zygmund theory and in particular J. Parcet's extension of the classical Calder\'on--Zygmund theorem (see Theorem \ref{Thm=Parcet} and \cite{Parcet}).

In the text we prove a somewhat stronger result in the terms of double operator integrals (see next section for definition), of which Conjecture \ref{Conjecture} is a corollary.

\begin{thm}\label{doi bound} If $A$ is a self-adjoint operator affiliated with a semifinite von Neumann algebra $\mathcal{M},$ and if $f:\mathbb{R}\to \mathbb{C}$ is Lipschitz then
$$\|T_{f^{[1]}}^{A,A}(V)\|_{1,\infty}\leq c_{abs}\|f'\|_{\infty}\|V\|_1,\quad V\in(L_1\cap L_2)(\mathcal{M}).$$
\end{thm}

Commutator estimate follows from the observation that the double operator integral $T_{f^{[1]}}^{A,A}([A,B])$ equals $[f(A),B].$ As explained in the proof of Theorem \ref{final theorem}, Lipschitz estimates follow from commutator ones.

\subsection*{Acknowledgements} Authors thank Javier Parcet for a detailed explanation of \cite{Parcet}.

\section{Preliminaries}

\subsection{General notation} Let $\mathcal{M}$ be a semifinite von Neumann algebra equipped with a faithful normal semifinite trace $\tau.$ In this paper, we always presume that $\mathcal{M}$ is represented on a separable Hilbert space.

A (closed and densely defined) operator $x$ affiliated to $\mathcal{M}$ is called $\tau-$measurable if $\tau(E_{|x|}(s,\infty))<\infty$ for sufficiently large $s.$ We denote the set of all $\tau-$measurable operators by $S(\mathcal{M},\tau).$ For every $x\in S(\mathcal{M},\tau),$ we define its singular value function $\mu(A)$ by setting
$$\mu(t,x)=\inf\{\|x(1-p)\|_{\infty}:\quad \tau(p)\leq t\}.$$
Equivalently, for positive operator $x\in S(\mathcal{M},\tau),$ we have
$$n_x(s)=\tau(E_x(s,\infty)),\quad \mu(t,x)=\inf\{s: n_x(s)<t\}.$$
We have (see e.g. \cite[Corollary 2.3.16]{LSZ})
\begin{equation}\label{triangle svf}
\mu(t+s,x+y)\leq\mu(t,x)+\mu(s,y),\quad t,s>0.
\end{equation}

\subsection{Non-commutative spaces}
For $1\leq p<\infty$ we set,
$$L_p(\mathcal{M})=\{x\in S(\mathcal{M},\tau):\ \tau(|x|^p)<\infty\},\quad \|x\|_p=(\tau(|x|^p))^{\frac1p}.$$
The Banach spaces $(L_p(\mathcal{M}),\|\cdot\|_p)$, $1\leq p<\infty$ are separable.

Define the space $L_{1,\infty}(\mathcal{M})$ by setting
$$L_{1,\infty}(\mathcal{M})=\{x\in S(\mathcal{M},\tau):\ \sup_{t>0}t\mu(t,x)<\infty\}.$$
We equip $L_{1,\infty}(\mathcal{M})$ with the functional $\|\cdot\|_{1,\infty}$ defined by the formula
$$\|x\|_{1,\infty}=\sup_{t>0}t\mu(t,x),\quad x\in L_{1,\infty}(\mathcal{M}).$$
It follows from \eqref{triangle svf} that
$$\|x+y\|_{1,\infty}=\sup_{t>0}t\mu(t,x+y)\leq\sup_{t>0}t(\mu(\frac{t}{2},x)+\mu(\frac{t}{2},y))\leq$$
$$\leq\sup_{t>0}t\mu(\frac{t}{2},x)+\sup_{t>0}t\mu(\frac{t}{2},y)=2\|x\|_{1,\infty}+2\|y\|_{1,\infty}.$$
In particular, $\|\cdot\|_{1,\infty}$ is a quasi-norm. The quasi-normed space $(L_{1,\infty}(\mathcal{M}),\|\cdot\|_{1,\infty})$ is, in fact, quasi-Banach (see e.g. \cite[Section 7]{KS} or \cite{Sind}). In view of our main result it is important to emphasize that the quasi-norm $\|\cdot\|_{1,\infty}$ is not equivalent to any norm on $L_{1,\infty}(\mathcal{M})$ (see e.g  \cite[Theorem 7.6]{KS}).

\subsection{Weak type inequalities for Calder\'on-Zygmund operators}
Parcet \cite{Parcet} proved a noncommutative extension of Calder\'on-Zygmund theory.

Let $K$ be a tempered distribution which we refer to as the {\it convolution kernel}. We let $W_K$ be the associated Calder\'on-Zygmund operator, formally given by $f \mapsto K \ast f.$ In what follows, we only consider tempered distributions having local values (that is, which can be identified with measurable functions $K:\mathbb{R}^d\rightarrow\mathbb{C}$).

Let $\mathcal{M}$ be a semi-finite von Neumann algebra with normal, semi-finite, faithful trace $\tau.$ The operator $1\otimes W_K$ can, under suitable conditions, be defined as noncommutative Calder\'on-Zygmund operators by letting them act on the second tensor leg of $L_1(\mathcal{M})\widehat{\otimes}L_1(\mathbb{R}^d).$ The following theorem in particular gives a sufficient condition for such an operator to act from $L_1$ to $L_{1,\infty}.$

\begin{thm}[\cite{Parcet}]\label{Thm=Parcet}
Let $K:\mathbb{R}^d \backslash \{0\} \rightarrow \mathbb{C}$ be a kernel satisfying the conditions\footnote{Here, $\nabla$ denotes the gradient $(\frac1i\frac{\partial}{\partial x_1},\cdots,\frac1i\frac{\partial}{\partial x_d}),$ which is understood as unbounded self-adjoint operator on $L_2(\mathbb{R}^d).$}
\begin{equation}\label{parcet conditions}
|K|(t)\leq\frac{{\rm const}}{|t|^d},\quad |\nabla K|(t)\leq\frac{{\rm const}}{|t|^{d+1}}.
\end{equation}
Let $\mathcal{M}$ be a semi-finite von Neumann algebra. If $W_K\in B(L_2(\mathbb{R}^d)),$ then the operator $1\otimes W_K$ defines a bounded map from $L_1(\mathcal{M}\otimes L_\infty(\mathbb{R}^d))$ to $L_{1,\infty}(\mathcal{M}\otimes L_\infty(\mathbb{R}^d)).$
\end{thm}

\subsection{Double operator integrals}\label{doi subsection} Let $A=A^*$ be an operator affiliated with $\mathcal{M}.$ Symbolically, a double operator integral is defined by the formula
\begin{equation}\label{doi def}
T_{\xi}^{A,A}(V)=\int_{\mathbb{R}^2}\xi(\lambda,\mu)dE_A(\lambda)VE_A(\mu),\quad V\in L_2(\mathcal{M}).
\end{equation}
In the subsequent paragraph, we provide a rigorous definition of the double operator integral.

Consider projection valued measures on $\mathbb{R}$ acting on the Hilbert space $L_2(\mathcal{M})$ by the formulae $x\to E_A(\mathcal{B})x$ and $x\to xE_A(\mathcal{B}).$ These spectral measures commute and, hence (see Theorem V.2.6 in \cite{BirSol}), there exists a countably additive (in the strong operator topology) projection-valued measure $\nu$ on $\mathbb{R}^2$ acting on the Hilbert space $L_2(\mathcal{M})$ by the formula
\begin{equation}\label{birsol spectral measure}
\nu(\mathcal{B}_1\otimes\mathcal{B}_2):x\to E_A(\mathcal{B}_1)xE_A(\mathcal{B}_2),\quad x\in L_2(\mathcal{M}).
\end{equation}
Integrating a bounded Borel function $\xi$ on $\mathbb{R}^2$ with respect to the measure $\nu$ produces a bounded operator acting on the Hilbert space $L_2(\mathcal{M}).$ In what follows, we denote the latter operator by $T_{\xi}^{A,A}$ (see also \cite[Remark 3.1]{PSW}).

In the special case when $A$ is bounded and ${\rm spec}(A)\subset\mathbb{Z},$ we have 
\begin{equation}\label{doi special}
T_{\xi}^{A,A}(V)=\sum_{k,l\in\mathbb{Z}}\xi(k,l)E_A(\{k\})VE_A(\{l\}).
\end{equation}

We are mostly interested in the case $\xi=f^{[1]}$ for a Lipschitz function $f.$ Here,
$$f^{[1]}(\lambda,\mu)=
\begin{cases}
\frac{f(\lambda)-f(\mu)}{\lambda-\mu},\quad \lambda\neq\mu\\
0,\quad \lambda=\mu.
\end{cases}
$$

\section{Approximate intertwining properties of Fourier multipliers}\label{Sect=DeLeeuw}

We prove intertwining properties of Fourier multipliers partly inspired by K. de Leeuw's proof of his restriction theorem for $L^p$-multipliers \cite[Section 2]{de Leeuw}.

In what follows,
$$G_l(s)=\frac1{l\sqrt{\pi}}e^{-(\frac{s}{l})^2},\quad s\in\mathbb{R}, \quad l>0.$$
That is, $G_l$ is a probability density function for certain Gaussian random variable. The notation $G_l^{\otimes d}$ stands for the function from $L_1(\mathbb{R}^d)$ given by the tensor product of $G_l$ with itself repeated $d$ times.

\begin{lem}\label{first convergence lemma} For every $f\in L_1(\mathbb{R})$ with $\int_{-\infty}^{\infty}f(s)ds=0,$ we have $f\ast G_l\to0$ in $L_1(\mathbb{R})$ as $l\to\infty.$
\end{lem}
\begin{proof} Suppose first that $f$ is a step function of the form $f=\sum_{k=1}^m\alpha_k\chi_{I_k},$ $m\ge 1$ where $I_k=[a_k,b_k],\ 1\leq k\leq m$ are disjoint intervals and $\sum_{k=1}^m\alpha_km(I_k)=0.$ 

We have
$$(f\ast G_l)(t)=\sum_{k=1}^m\alpha_k\int_{a_k}^{b_k}G_l(t-s)ds=\sum_{k=1}^m\alpha_k\int_{t-b_k}^{t-a_k}G_l(u)du=$$
$$=\sum_{k=1}^m\alpha_k\int_{\frac{t-b_k}{l}}^{\frac{t-a_k}{l}}G_1(s)ds=\sum_{k=1}^m\alpha_k(F(\frac{t-a_k}{l})-F(\frac{t-b_k}{l})),$$
where $F(t)=\int_{-\infty}^tG_1(s)ds.$ To prove the assertion for our $f,$ it suffices to show that
$$l\int_{-\infty}^{\infty}\left|\sum_{k=1}^m\alpha_k(F(t-\frac{a_k}{l})-F(t-\frac{b_k}{l}))\right|dt\to0$$
Clearly,
$$\left|F(t-\frac{a_k}{l})-F(t)+\frac{a_k}{l}F'(t)\right|\leq\frac{a_k^2}{2l^2}\max\limits_{s\in[t-\frac{a_k}{l},t]}|F''(s)|.$$

If $l>\max_{1\leq k\leq m}|a_k|$ and $l>\max_{1\leq k\leq m}|b_k|,$ then 
$$\left|\sum_{k=1}^m\alpha_k(F(t-\frac{a_k}{l})-F(t-\frac{b_k}{l}))\right|\leq\frac1{2l^2}(\sum_{k=1}^m|\alpha_k|(a_k^2+b_k^2))\max\limits_{s\in[t-1,t+1]}|F''(s)|.$$
This proves the assertion for $f$ as above.

To prove the assertion in general, fix $f_m$ as above (i.e., mean zero step functions) such that $f_m\to f$ in $L_1(\mathbb{R}).$ Since $\|G_l\|_1=1,$ it follows from Young's inequality that
$$\|f\ast G_l\|_1\leq\|(f-f_m)\ast G_l\|_1+\|f_m\ast G_l\|_1\leq \|f-f_m\|_1+\|f_m\ast G_l\|_1.$$
Therefore,
$$\limsup_{l\to\infty}\|f\ast G_l\|_1\leq\|f-f_m\|_1.$$
Passing $m\to\infty,$ we conclude the proof.
\end{proof}

By Fubini Theorem, linear span of elementary tensors
$$(f_1\otimes\cdots\otimes f_d):(t_1,\cdots,t_d)\to  f_1(t_1)\cdots f_d(t_d),\quad f_1,\cdots,f_d\in L_1(\mathbb{R})$$
is dense in $L_1(\mathbb{R}^d).$

\begin{lem}\label{second convergence lemma} For every $f\in L_1(\mathbb{R}^d)$ with $\int_{\mathbb{R}^d}f(s)ds=0,$ we have $f\ast G_l^{\otimes d}\to0$ in $L_1(\mathbb{R}^d)$ as $l\to\infty.$
\end{lem}
\begin{proof} Suppose first that $f$ is a linear combination of elementary tensors. That is,
\begin{equation}\label{f first decom}
f=\sum_{k=1}^m\bigotimes_{j=1}^df_{jk},\quad f_{jk}\in L_1(\mathbb{R}).
\end{equation}
Firstly, we consider the case when for every $k$, $1\leq k\leq m$ there exists $j$, $1\leq j\leq d$ such that $\int_{\mathbb{R}}f_{jk}(s)=0.$ In this case, by Lemma \ref{first convergence lemma} we have that
$$\|f\ast G_l^{\otimes d}\|_1\leq\sum_{k=1}^m\prod_{j=1}^d\|f_{jk}\ast G_l\|_1\to0.$$
Now, we show that the case of $f$ given by \eqref{f first decom} satisfying
\begin{equation}\label{mean zero cond}
\sum_{k=1}^m\prod_{j=1}^d\int_{\mathbb{R}}f_{jk}(s)ds=0.
\end{equation}
can be reduced to the just considered case when for every $k$, $1\leq k\leq m$ there exists $j$, $1\leq j\leq d$ such that 
$\int_{\mathbb{R}}f_{jk}(s)=0.$ To this end, 
for every subset $\mathscr{A}\subset\{1,\cdots,d\},$ we set
$$f_{j,k,\mathscr{A}}=
\begin{cases}
f_{jk}-(\int_{\mathbb{R}}f_{jk}(s)ds)\chi_{(0,1)},\quad j\in\mathscr{A}\\
(\int_{\mathbb{R}}f_{jk}(s)ds)\chi_{(0,1)},\quad j\notin\mathscr{A}.
\end{cases}
$$
By the linearity, we can rewrite \eqref{f first decom} as 
$$
f=\sum_{k=1}^m\sum_{\mathscr{A}\subset\{1,\cdots,d\}}\bigotimes_{j=1}^df_{j,k,\mathscr{A}}
$$
Observing now that
$$
\sum_{k=1}^m\bigotimes_{j=1}^df_{j,k,\varnothing}=(\sum_{k=1}^m\prod_{j=1}^d\int_{\mathbb{R}}f_{jk}(s)ds)\chi_{(0,1)}^{\otimes d}
$$
and appealing to \eqref{mean zero cond}, we arrive at 
\begin{equation}\label{f second decom}
f=\sum_{k=1}^m\sum_{\varnothing\neq\mathscr{A}\subset\{1,\cdots,d\}}\bigotimes_{j=1}^df_{j,k,\mathscr{A}}.
\end{equation}

Note that $f_{j,k,\mathscr{A}}$ is mean zero for $j\in\mathscr{A}.$ Using representation \eqref{f second decom} instead of \eqref{f first decom} for $f$, we may assume without loss of generality that for every $k$, $1\leq k\leq m$ there exists $j$, $1\leq j\leq d$ such that $\int_{\mathbb{R}}f_{jk}(s)=0.$ This completes the proof of the lemma in the special case when $f$ is given by \eqref{f first decom} and satisfies \eqref{mean zero cond}.

To prove the general case, fix  $f\in L_1(\mathbb{R}^d)$ with $\int_{\mathbb{R}^d}f(s)ds=0,$ and select a sequence $\{f_m\}_{m=1}^\infty$ of mean zero sums of elementary tensors such that $f_m\to f$ in $L_1(\mathbb{R}^d)$ as $m\to \infty$. Since $\|G_l^{\otimes d}\|_1=1,\ l\ge 1$ it follows from Young inequality that
$$\|f\ast G_l^{\otimes d}\|_1\leq\|(f-f_m)\ast G_l^{\otimes d}\|_1+\|f_m\ast G_l^{\otimes d}\|_1\leq \|f-f_m\|_1+\|f_m\ast G_l^{\otimes d}\|_1.$$
Therefore,
$$\limsup_{l\to\infty}\|f\ast G_l^{\otimes d}\|_1\leq\|f-f_m\|_1.$$
Passing $m\to\infty,$ we conclude the proof.
\end{proof}

In what follows,
\begin{equation}\label{ek def}
e_k(t):=e^{i\langle k,t\rangle},\quad k,t\in\mathbb{R}^d
\end{equation}
and $\mathcal{F}$ stands for the Fourier transform.

\begin{lem}\label{third convergence lemma} If $g\in L_\infty(\mathbb{R}^d)$ is such that $\mathcal{F}(g)\in L_1(\mathbb{R}^d),$ then for every $k\in\mathbb{R}^d$ we have
$$(g(\nabla))(G_l^{\otimes d}e_k)- g(k)G_l^{\otimes d}e_k\to0$$
in $L_1(\mathbb{R}^d)$ as $l\to\infty.$
\end{lem}
\begin{proof} Fix $k\in\mathbb{R}^d.$ Set $h_1(t): =g(k)e^{-|t-k|^2}$ and $h_0(t):=g(t)-h_1(t)$,  $t\in\mathbb{R}^d.$ Observe that, for every $t\in\mathbb{R}^d,$ we have
$$\mathcal{F}(G_l^{\otimes d})(t)=\pi^{-d/2}e^{-l^2|t|^2}.$$
Since $h_1(\nabla)$ on the Fourier side is a multiplier on $h_1,$ it follows that, for every $t\in \mathbb{R}^d,$ 
$$\mathcal{F}(G_l^{\otimes d}e_k)(t)=\pi^{-d/2}e^{-l^2|t-k|^2},\quad (\mathcal{F}((h_1(\nabla))(G_l^{\otimes d}e_k)))(t)=g(k)\pi^{-d/2}e^{-(l^2+1)|t-k|^2}.$$
Applying the inverse Fourier transform to the second equality, we arrive at
$$(h_1(\nabla))(G_l^{\otimes d}e_k)=g(k)G_{(l^2+1)^{1/2}}^{\otimes d}e_k.$$
A direct computation yields $G_{(l^2+1)^{1/2}}^{\otimes d}-G_l^{\otimes d}\to 0$ in $L_1(\mathbb{R}^d)$ as $l\to\infty.$ We conclude that
$$(h_1(\nabla))(G_l^{\otimes d}e_k)-g(k)G_l^{\otimes d}e_k\to 0$$
in $L_1(\mathbb{R}^d)$ as $l\to\infty.$ It, therefore, suffices to show that
$$(h_0(\nabla))(G_l^{\otimes d}e_k)\to 0$$
in $L_1(\mathbb{R}^d)$ as $l\to\infty.$ Define the function $f\in L_1(\mathbb{R}^d)$ by setting $f(t)=e^{i\langle k,t\rangle}(\mathcal{F}h_0)(t),$ $t\in\mathbb{R}^d.$ We rewrite the latter equation as $f\ast G_l^{\otimes d}\to0$ as $l\to\infty.$ Note that
$$\int_{\mathbb{R}^d}f(s)ds=\int_{\mathbb{R}^d}e^{i\langle k,s\rangle}(\mathcal{F}h_0)(s)ds=h_0(k)=0.$$
The assertion follows now from the Lemma \ref{second convergence lemma}.
\end{proof}

\section{Proof of Theorem \ref{doi bound} in the special case}

For $s>0,$ the dilation operator $\sigma_s$ acts on the space of Lebesgue measurable functions on $\mathbb{R},$ by the formula $(\sigma_sx)(t)=x(t/s).$

\begin{lem}\label{tensor mu} Let  $x,y$ be measurable and $\theta$ be integrable functions on $\mathbb{R}.$ Let $z(t):=t^{-1},$ $t>0,$ $z(t)=0,$ $t<0,$ and let $u>0$. For Lebesgue measurable functions 
$x\otimes y$ and $\theta \otimes z$ on $\mathbb{R}^2,$ we have
$$\mu(\sigma_u(x)\otimes y)=\sigma_u\mu(x\otimes y),  \quad \mu(t,\theta \otimes z)=\|\theta\|_1t^{-1},\quad t>0.$$
\end{lem}
\begin{proof} Denoting Lebesgue measure on $\mathbb{R}^2$ by $m$, we have for every $t>0$
\begin{align*}m(\{\sigma_u(x)\otimes y>t \})&=m(\{(s_1,s_2):\ x(\frac{s_1}{u})y(s_2)>t\})\cr
&=um(\{(s_1,s_2):\ x(s_1)y(s_2)>t\})\cr & =um(\{x\otimes y>t \}).
\end{align*}
This proves the first assertion.

Firstly, we prove the second assertion for simple function $x\in L_1(\mathbb{R}).$  If $x=\sum_ka_k\chi_{B_k}$ with $B_k$ being pairwise disjoint sets, then\footnote{The notation $\bigoplus_k x_k$ stands for disjoint sum of the functions $x_k,$ that is $\sum_k z_k,$ where functions $z_k$ have pairwise disjoint support and $\mu(z_k)=\mu(x_k).$ We refer the reader to the Definition 2.4.3 in \cite{LSZ} and subsequent comments.}
$$\mu(x\otimes z)=\mu(\bigoplus_k (a_k\chi_{B_k}\otimes z))=\mu(\bigoplus_k\mu((a_k\chi_{B_k}\otimes z)).$$
If $B$ is a set of finite measure, then there exists a measure preserving bijection from $B$ to $(0,m(B))$ (see \cite{HN}). Therefore, we have
$$\mu(\chi_B\otimes z)=\mu(\chi_{(0,m(B))}\otimes z)=m(B)z.$$
Thus,
$$\mu(x\otimes z)=\mu(\bigoplus_k|a_k|m(B_k)z)=(\sum_k a_km(B_k))z.$$
The second assertion follows now by approximation.
\end{proof}

\begin{lem}\label{gl tensor lemma} For every $X\in L_{1,\infty}(\mathcal{M})$ and every $l>0$, we have
$$e^{-d}\pi^{-\frac{d}{2}}\|X\|_{1,\infty}\leq\|X\otimes G_l^{\otimes d}\|_{1,\infty}\leq\|X\|_{1,\infty}.$$
\end{lem}
\begin{proof} For every operator $A\in S(\mathcal{M},\tau)$ and for every function $g\in L_{\infty}(0,\infty),$ we have\footnote{Without loss of generality, $\mathcal{M}$ is atomless. Suppose first that $x\in\mathcal{M}$ is $\tau-$compact. By Theorem 2.3.11 in \cite{LSZ}, there exists a trace preserving $*-$isomorphism $i:L_{\infty}(0,\infty)\to \mathcal{M}_1$ such that $i_1(\mu(A))=|A|.$ Consider trace preserving isomorphism $i\otimes 1:L_{\infty}(0,\infty)\otimes L_{\infty}(0,\infty)\to\mathcal{M}\otimes L_{\infty}(0,\infty).$ We have $i(\mu(A)\otimes g)=|A|\otimes g.$ Since every trace preserving $*-$isomorphism preserves singular value function, the claim follows for $\tau-$compact operators. The general case follows by approximation.}
\begin{equation}\label{mu tensor}
\mu(A\otimes g)=\mu(|A|\otimes g)=\mu(\mu(A)\otimes g)=\mu(\mu(A)\otimes \mu(g)).
\end{equation} 
Let $z$ be as in Lemma \ref{tensor mu}. It follows from the definition $\|\cdot\|_{1,\infty}$ that $\mu(X)\leq\|X\|_{1,\infty}z$ and, therefore,
$$\mu(X\otimes G_l^{\otimes d})\stackrel{\eqref{mu tensor}}{=}\mu(\mu(X)\otimes G_l^{\otimes d})\leq\|X\|_{1,\infty}\mu(z\otimes G_l^{\otimes d})\stackrel{L.\ref{tensor mu}}{=}\|X\|_{1,\infty}\mu(z).$$
This proves the right hand side inequality.

On the other hand, $\mu(G_l)=l^{-1}\sigma_l\mu(G).$ By Lemma \ref{tensor mu}, we have $\mu(G_l^{\otimes d})=l^{-d}\sigma_{l^d}\mu(G_1^{\otimes d}).$ Thus,
$$\mu(X\otimes G_l^{\otimes d})\stackrel{\eqref{mu tensor}}{=}\mu(X\otimes l^{-d}\sigma_{l^d}\mu(G_1^{\otimes d}))\stackrel{L.\ref{tensor mu}}{=}l^{-d}\sigma_{l^d}\mu(X\otimes G_1^{\otimes d}).$$
Therefore, we have
$$\|X\otimes G_l^{\otimes d}\|_{1,\infty}=\sup_{t>0}\frac{t}{l^d}\mu(\frac{t}{l^d},X\otimes G_1^{\otimes d})=\sup_{s>0}s\mu(s,X\otimes G_1^{\otimes d})=\|X\otimes G_1^{\otimes d}\|_{1,\infty}.$$
Clearly, $\mu(G_1)\geq\frac1{e\sqrt{\pi}}\chi_{(0,1)}.$ It follows that
$$\|X\otimes G_1^{\otimes d}\|_{1,\infty}\stackrel{\eqref{tensor mu}}{=}\|X\otimes \mu(G_1)^{\otimes d}\|_{1,\infty}\geq\|X\otimes (\frac1{e\sqrt{\pi}}\chi_{(0,1)})^{\otimes d}\|_{1,\infty}=e^{-d}\pi^{-\frac{d}{2}}\|X\|_{1,\infty}.$$
This proves the left hand side inequality.
\end{proof}

The following lemma is ideologically similar to Theorem II.4.3 in \cite{Stein}.

\begin{lem}\label{g is ok} If $g$ is a smooth homogeneous function on $\mathbb{R}^2\backslash\{0\},$ then $\mathcal{F}g$ satisfies (possibly, after some $\delta$ distribution is subtracted) the conditions \eqref{parcet conditions}.
\end{lem}
\begin{proof} By assumption, $g$ is a smooth function on the circle $\{|z|=1\}.$ Thus,
$$g(e^{i\theta})=\sum_{k\in\mathbb{Z}}\alpha_k e^{ik\theta},$$
where Fourier coefficients decrease faster than every power. Therefore,
$$g=\sum_{k\in\mathbb{Z}}\alpha_k g_k,\quad g_k(z)=\frac{z^k}{|z|^k},\quad 0\neq z\in\mathbb{C}.$$
For every $k\neq0,$ we have\footnote{This can be checked e.g. by substituting $m=\Omega=g_k$ into the formula (26) in Theorem II.4.3 in \cite{Stein}.}
$$(\mathcal{F}g_k)(z)=\frac{|k|}{2\pi i^k}\cdot\frac{g_k(z)}{|z|^2},\quad 0\neq z\in\mathbb{C}.$$
Hence,
$$(\mathcal{F}g)(z)=\alpha_0\delta+\frac1{|z|^2}h(e^{i{\rm Arg}(z)}),$$
where the smooth function $h$ on the circle is defined by the formula
$$h(e^{i\theta})=\sum_{0\neq k\in\mathbb{Z}}\frac{|k|}{2\pi i^k}\alpha_k e^{ik\theta}.$$
So, $(\mathcal{F}g-\alpha_0\delta)(z)=O(|z|^{-2}).$ Furthermore, have
$$\nabla(\frac{h(e^{i{\rm Arg(z)}})}{|z|^2})=h(e^{i{\rm Arg(z)}})\cdot\nabla(\frac1{|z|^2})+\frac1{|z|^2}\cdot \frac{dh(e^{i\theta})}{d\theta}|_{\theta={\rm Arg}(z)}\cdot \nabla({\rm Arg}(z))=O(\frac1{|z|^3}).$$
This completes the verification that $\mathcal{F}g-\alpha_0\delta$ satisfies condition \eqref{parcet conditions}.
\end{proof}

\begin{thm}\label{main lemma} For every $A=A^*\in\mathcal{M}$ with ${\rm spec}(A)\subset\mathbb{Z}$ and for every Lipschitz function $f,$ we have
$$\|T_{f^{[1]}}^{A,A}(V)\|_{1,\infty}\leq c_{abs}\|f'\|_{\infty}\|V\|_1,\quad V\in L_1(\mathcal{M}).$$
\end{thm}
\begin{proof} Fix a smooth homogeneous function $g$ on $\mathbb{R}^2$ such that $g(e^{i\theta})=\tan(\theta)$ for $\theta\in(-\frac{\pi}{4},\frac{\pi}{4})$ and for $\theta\in(\frac{3\pi}{4},\frac{5\pi}{4}).$ Without loss of generality, $g$ is mean zero on the circle $\{|z|=1\}.$ By Lemma \ref{g is ok}, $\mathcal{F}g$ satisfies the conditions \eqref{parcet conditions}. The operator $g(\nabla)\in B(L_2(\mathbb{R}^2))$ since $g$ is bounded. Recall that $(g(\nabla))(x)=(\mathcal{F}g)\ast x.$  By Theorem \ref{Thm=Parcet}, we have
$$1\otimes g(\nabla):L_1(\mathcal{M}\otimes L_\infty(\mathbb{R}^2))\to L_{1,\infty}(\mathcal{M}\otimes L_\infty(\mathbb{R}^2)).$$
Consider Schwartz functions\footnote{Let $\psi$ be a Schwartz function on $\mathbb{R}$ which is $1$ on $(-1,1)$ and which is supported on $(-2,2).$ Set $\psi_m=\sigma_m\psi\cdot(1-\sigma_{\frac1m}\psi).$ It follows that
$$\mathcal{F}(\psi_m)=\mathcal{F}(\sigma_m\psi)\ast\mathcal{F}(1-\sigma_{\frac1m}\psi)=m\sigma_{\frac1m}(\mathcal{F}(\psi))-m\sigma_{\frac1m}(\mathcal{F}(\psi))\ast \frac1m\sigma_m(\mathcal{F}(\psi)).$$
Applying Young's inequality, we conclude that
$$\|\mathcal{F}(\psi_m)\|_1\leq \|m\sigma_{\frac1m}(\mathcal{F}(\psi))\|_1+\|m\sigma_{\frac1m}(\mathcal{F}(\psi))\|_1\|\frac1m\sigma_m(\mathcal{F}(\psi))\|_1=\|\mathcal{F}(\psi)\|_1+\|\mathcal{F}(\psi)\|_1^2.$$
Consider the functions $\phi_m=\psi_{3m}^{\otimes 2}.$ By Fubini Theorem, $\sup_{m\geq1}\|\mathcal{F}(\phi_m)\|_1<\infty.$ Clearly, $\psi_m=1$ on the set $[-m,m]\backslash[-\frac2m,\frac2m].$ Thus, $\phi_m(t)=1$ if $t\in 3mK$ and $3mt\notin 2K,$ where $K=[-1,1]\times[-1,1].$ Thus, $\phi_m(t)=1$ whenever $|t|\in(\frac1m,m).$} $\phi_m$ on $\mathbb{R}^2$ which vanish near $0,$ such that $\phi_m(t)=1$ for $|t|\in(\frac1m,m)$ and such that $\|\mathcal{F}\phi_m\|_1\leq c_{abs}$ for all $m\geq1.$ It follows that
$$\|1\otimes (g\phi_m)(\nabla)\|_{L_1\to L_{1,\infty}}\leq\|1\otimes g(\nabla)\|_{L_1\to L_{1,\infty}}\|1\otimes\phi_m(\nabla)\|_{L_1\to L_1}\leq$$
$$\leq \|1\otimes g(\nabla)\|_{L_1\to L_{1,\infty}}\|\mathcal{F}\phi_m\|_1\leq c_{abs}\|1\otimes g(\nabla)\|_{L_1\to L_{1,\infty}}=c_{abs},\quad m\geq1.$$
The last equality holds because $g$ is fixed.
 
By assumption, $A=\sum_{j\in\mathbb{Z}}jp_j,$ where $\{p_j\}_{j\in\mathbb{Z}}$ are pairwise orthogonal projections such that $\sum_{j\in\mathbb{Z}}p_j=1.$ Since $A$ is bounded, it follows that $p_j=0$ for all but finitely many $j\in\mathbb{Z}.$ Hence, sums are, in fact, finite. Consider a unitary operator
$$u=\sum_{j\in\mathbb{Z}}p_j\otimes e_{(j,f(j))},$$
where $e_{(j,f(j))}$ is given in \eqref{ek def}.
Without loss of generality, $\|f'\|_{\infty}\leq 1.$ For every $m\geq\|A\|_{\infty},$ we have $|i-j|,|f(i)-f(j)|\leq 2m$ for every $i,j\in{\rm spec}(A).$ Hence,
$$(g\phi_m)(i-j,f(i)-f(j))=g(i-j,f(i)-f(j))=\frac{f(i)-f(j)}{i-j},\quad i,j\in{\rm spec}(A),\quad i\neq j.$$
It follows from the preceding paragraph and from the equality $\|G_l^{\otimes 2}\|_1=1$ that
\begin{equation}\label{form0}
\|(1\otimes(g\phi_m)(\nabla))(u(V\otimes G_l^{\otimes 2})u^*)\|_{1,\infty}\leq c_{abs}\|u(V\otimes G_l^{\otimes 2})u^*\|_1=c_{abs}\|V\|_1.
\end{equation}
 
It is clear that
$$(1\otimes(g\phi_m)(\nabla))(u(V\otimes G_l^{\otimes 2})u^*)=\sum_{i,j}p_iVp_j\otimes (g\phi_m(\nabla))(G_l^{\otimes 2}e_{(i-j,f(i)-f(j))}).$$
Since there are only finitely many summands, it follows from Lemma \ref{third convergence lemma} (as applied to the Schwartz function $g\phi_m$) that
$$(1\otimes(g\phi_m)(\nabla))(u(V\otimes G_l^{\otimes 2})u^*)-\sum_{i\neq j}p_iVp_j\otimes\frac{f(i)-f(j)}{i-j}G_l^{\otimes 2}e_{(i-j,f(i)-f(j))}\to0$$
in $L_1(\mathcal{M}\otimes L_{\infty}(\mathbb{R}^2))$ as $l\to\infty.$ It is immediate that
$$\sum_{i\neq j}p_iVp_j\otimes\frac{f(i)-f(j)}{i-j}G_l^{\otimes 2}e_{(i-j,f(i)-f(j))}=$$
$$=\Big(\sum_{k\in\mathbb{Z}}p_k\otimes e_{(k,f(k))}\Big)\cdot\Big(\sum_{i\neq j}p_iVp_j\otimes\frac{f(i)-f(j)}{i-j}G_l^{\otimes 2}\Big)\cdot \Big(\sum_{l\in\mathbb{Z}}p_l\otimes e_{(-l,-f(l))}\Big)=$$
$$\stackrel{\eqref{doi special}}{=}u(T_{f^{[1]}}^{A,A}(V)\otimes G_l^{\otimes 2})u^*.$$
Therefore,
\begin{equation}\label{form1}
(1\otimes(g\phi_m)(\nabla))(u(V\otimes G_l^{\otimes 2})u^*)-u(T_{f^{[1]}}^{A,A}(V)\otimes G_l^{\otimes 2})u^*\to0
\end{equation}
in $L_1(\mathcal{M}\otimes L_{\infty}(\mathbb{R}^2))$ (and, hence, in $L_{1,\infty}(\mathcal{M}\otimes L_{\infty}(\mathbb{R}^2))$) as $l\to\infty.$

Combining \eqref{form0} and \eqref{form1}, we arrive at
$$\limsup_{l\to\infty}\|u(T_{f^{[1]}}^{A,A}(V)\otimes G_l^{\otimes 2})u^*\|_{1,\infty}\leq c_{abs}\|V\|_1.$$
Since $u$ is unitary, it follows that
$$\limsup_{l\to\infty}\|T_{f^{[1]}}^{A,A}(V)\otimes G_l^{\otimes 2}\|_{1,\infty}\leq c_{abs}\|V\|_1.$$
The assertion follows now from Lemma \ref{gl tensor lemma}.
\end{proof}

\section{Proof of the main results}

In this section we collect the results announced in the abstract and its corollaries. Throughout this section fix a semi-finite von Neumann algebra $\mathcal{M}$ with normal, semi-finite, faithful trace $\tau.$

\begin{lem}\label{riemann} Let $A=A^*\in\mathcal{M}.$ If $\{\xi_n\}_{n\geq0}$ is a uniformly bounded sequence of Borel functions on $\mathbb{R}^2$ such that $\xi_n\to\xi$ everywhere, then
\begin{equation}\label{riemann sum convergence}
T_{\xi_n}^{A,A}(V)\to T_{\xi}^{A,A}(V),\quad V\in L_2(\mathcal{M})
\end{equation}
in $L_2(\mathcal{M})$ as $n\to\infty.$
\end{lem}
\begin{proof} Let $\nu$ be a projection valued measure on $\mathbb{R}^2$ considered in Subsection \ref{doi subsection} (see \eqref{birsol spectral measure}). Let $\gamma:\mathbb{R}\to\mathbb{R}^2$ be a Borel measurable bijection. Clearly, $\nu\circ\gamma$ is a projection valued measure on $\mathbb{R}.$ Hence, there exists a self-adjoint operator $B$ acting on the Hilbert space $L_2(\mathcal{M})$ such that $E_B=\nu\circ\gamma.$

Set $\eta_n=\xi_n\circ\gamma$ and $\eta=\xi\circ\gamma.$ We have $\eta_n\to\eta$ everywhere on $\mathbb{R}.$ Thus, 
$$T_{\xi_n}^{A,A}=\int_{\mathbb{R}^2}\xi_nd\nu=\int_{\mathbb{R}}\eta_n(\lambda)dE_B(\lambda)=\eta_n(B)\to\eta(B)=$$
$$=\int_{\mathbb{R}}\eta(\lambda)dE_B(\lambda)=\int_{\mathbb{R}^2}\xi d\nu=T_{\xi}^{A,A}.$$
Here, the convergence is understood with respect to the strong operator topology on the space $B(L_2(\mathcal{M})).$ In particular, \eqref{riemann sum convergence} follows.
\end{proof}

\begin{proof}[Proof of Theorem \ref{doi bound}] {\bf Step 1.} Let $A$ is bounded. For every $n\geq1,$ set
$$A_n\stackrel{def}{=}\sum_{k\in\mathbb{Z}}\frac{k}{n}E_A([\frac{k}{n},\frac{k+1}{n})),$$
$$\xi_n(t,s)=f^{[1]}(\frac{k}{n},\frac{l}{n}),\quad t\in[\frac{k}{n},\frac{k+1}{n}),s\in [\frac{l}{n},\frac{l+1}{n}).$$
It is immediate that (see e.g. Lemma 8  in \cite{PScrelle} for much stronger assertion)
$$T_{\xi_n}^{A,A}(V)=T_{f^{[1]}}^{A_n,A_n}(V)=T_{(n\sigma_nf)^{[1]}}^{nA_n,nA_n}(V).$$
It follows from Theorem \ref{main lemma} that
$$\|T_{\xi_n}^{A,A}(V)\|_{1,\infty}\leq c_{abs}\|(n\sigma_nf)'\|_{\infty}\|V\|_1=c_{abs}\|f'\|_{\infty}\|V\|_1.$$
Note that $\xi_n\to f^{[1]}$ everywhere. It follows from Lemma \ref{riemann} that
$$T_{\xi_n}^{A,A}(V)\to T_{f^{[1]}}(V),\quad V\in L_2(\mathcal{M})$$
in $L_2(\mathcal{M})$ (and, hence, in measure --- see e.g \cite{PSW}) as $n\to\infty.$ Since the quasi-norm in $L_{1,\infty}(\mathcal{M})$ is a Fatou quasi-norm, it follows that
$$\|T_{f^{[1]}}^{A,A}(V)\|_{1,\infty}\leq c_{abs}\|f'\|_{\infty}\|V\|_1,\quad V\in (L_1\cap L_2)(\mathcal{M}).$$

{\bf Step 2.} Let now $A$ be an arbitrary operator affiliated with $\mathcal{M}.$ Set $A_n=AE_A([-n,n]).$ By Step 1, we have
$$\|T_{f^{[1]}}^{A_n,A_n}(V)\|_{1,\infty}\leq c_{abs}\|f'\|_{\infty}\|V\|_1.$$
It follows immediately from the definition of the double operator integral that
$$T_{f^{[1]}}^{A_n,A_n}(V)=E_A([-n,n])\cdot T_{f^{[1]}}^{A,A}(V)\cdot E_A([-n,n])\to T_{f^{[1]}}^{A,A}(V)$$
in $L_2(\mathcal{M})$ (and, hence, in measure) as $n\to\infty.$ Since the quasi-norm in $L_{1,\infty}(\mathcal{M})$ is a Fatou quasi-norm, the assertion follows.
\end{proof}

The following lemma is ideologically similar to Theorem 7.4 in \cite{PSW}.

\begin{lem}\label{eg psw lemma} If $A,B\in\mathcal{M}$ are such that $[A,B]\in L_2(\mathcal{M}),$ then, for every Lipschitz function $f,$ we have
$$T_{f^{[1]}}^{A,A}([A,B])=[f(A),B].$$
\end{lem}
\begin{proof} By definition of double operator integral given in Subsection \ref{doi subsection}, we have
\begin{equation}\label{doi mult}
T_{\xi_1}^{A,A}T_{\xi_2}^{A,A}=T_{\xi_1\xi_2}^{A,A}.
\end{equation}
Let $\xi_1=f^{[1]}$ and let $\xi_2(\lambda,\mu)=\lambda-\mu$ when $|\lambda|,|\mu|\leq\|A\|_{\infty}.$ $\xi_2(\lambda,\mu)=0$ when $|\lambda|>\|A\|_{\infty}$ or $|\mu|\leq\|A\|_{\infty}.$

If $p$ is a $\tau-$finite projection, then $pB\in L_2(\mathcal{M})$ and
$$T_{\xi_1\xi_2}^{A,A}(pB)=f(A)pB-pBf(A),\quad T_{\xi_2}^{A,A}(pB)=ApB-pBA,$$
Applying \eqref{doi mult} to the operator $pB\in L_2(\mathcal{M}),$ we obtain
\begin{equation}\label{finite p}
T_{f^{[1]}}^{A,A}(ApB-pBA)=f(A)pB-pBf(A).
\end{equation}

Applying Proposition 6.6 in \cite{PSW} to the operator $nA,$ we construct a sequence $\{p_{n,k}\}_{k\geq0}$ of $\tau-$finite projections such that $p_{n,k}\uparrow 1$ as $k\to\infty$ and such that $\|[nA,p_{n,k}]\|_2\leq 1.$ Let $\{\eta_m\}_{m\geq0}$ be an orthonormal basis in $L_2(\mathcal{M}).$ Fix $k_n$ so large that
\begin{equation}\label{golova}
\|(1-p_{n,k_n})\eta_m\|_2\leq\frac1n,\quad 0\leq m<n.
\end{equation}
Set $q_n=p_{n,k_n}.$ It follows from \eqref{golova} that $q_n\to1$ in the strong operator topology (in the left regular representation of $\mathcal{M}$). Clearly, $[A,q_n]\to0$ in $L_2(\mathcal{M}).$

By construction,
$$Aq_nB-q_nBA=[A,q_n]B+q_n[A,B]\to[A,B],\quad n\to\infty,$$
in $L_2(\mathcal{M}).$ Since $T_{f^{[1]}}^{A,A}$ is bounded, it follows that
$$f(A)q_nB-q_nBf(A)\stackrel{\eqref{finite p}}{=}T_{f^{[1]}}^{A,A}(Aq_nB-q_nBA)\to T_{f^{[1]}}^{A,A}(AB-BA),\quad n\to\infty,$$
in $L_2(\mathcal{M}).$ On the other hand,
$$f(A)q_nB-q_nBf(A)\to [f(A),B],\quad n\to\infty,$$
in the strong operator topology (in the left regular representation of $\mathcal{M}$). This concludes the proof.
\end{proof}

\begin{thm}\label{final theorem} For all self-adjoint operators $A,B\in\mathcal{M}$ such that $[A,B]\in L_1(\mathcal{M})$ and for every Lipschitz function $f,$ we have
$$\|[f(A),B]\|_{1,\infty}\leq c_{abs}\|f'\|_{\infty}\cdot\|[A,B]\|_1.$$
For all self-adjoint operators $X,Y\in\mathcal{M}$ such that $X-Y\in L_1(\mathcal{M})$ and for every Lipschitz function $f,$ we have
$$\|f(X)-f(Y)\|_{1,\infty}\leq c_{abs}\|f'\|_{\infty}\|X-Y\|_1.$$
\end{thm}
\begin{proof} By assumption, $[A,B]\in (L_1\cap L_2)(\mathcal{M}).$ The first assertion follows by combining Lemma \ref{eg psw lemma} and Theorem \ref{doi bound}. Applying the first assertion to the operators
$$
A=
\begin{pmatrix}
X&0\\
0&Y
\end{pmatrix},\quad
B=
\begin{pmatrix}
0&1\\
1&0
\end{pmatrix},
$$
we obtain the second assertion.
\end{proof}


\begin{thebibliography}{99}
\bibitem{BiSo1} M. S. Birman and M. Z. Solomyak, Double Stieltjes operator
integrals  (Russian), Probl. Math. Phys.,  Izdat. Leningrad.
Univ., Leningrad, (1966)  33-–67. English translation in: Topics
in Mathematical Physics, Vol. 1 (1967),  Consultants Bureau
Plenum Publishing Corporation, New York, 25–-54.
\bibitem{BirSol} Birman M., Solomyak M. {\it Spectral theory of selfadjoint operators in Hilbert space.} Mathematics and its Applications (Soviet Series). D. Reidel Publishing Co., Dordrecht, 1987.
\bibitem{CMPS} Caspers M., Montgomery-Smith S., Potapov D., Sukochev F. {\it The best constants for operator Lipschitz functions on Schatten classes.} J. Funct. Anal. {\bf 267} (2014), no. 10, 3557--3579.
\bibitem{CPSZ} Caspers M., Potapov D., Sukochev F., Zanin D. {\it Weak type estimates for the absolute value mapping.} J. Operator Theory, {\bf 73} (2015), no. 2, 101--124.
\bibitem{Davies} Davies E. {\it Lipschitz continuity of functions of operators in the Schatten classes.} J. Lond. Math. Soc. {\bf 37} (1988) 148--157.
\bibitem{de Leeuw} de Leeuw, K. {\it On $L_p$ multipliers}. Ann. of Math. (2) {\bf 81 } (1965) 364--379. 
\bibitem{DDPS1} Dodds P., Dodds T., de Pagter B., Sukochev F. {\it Lipschitz continuity of the absolute value and Riesz projections in symmetric operator spaces.}   J. Funct. Anal. {\bf 148} (1997), 28--69.
\bibitem{DDPS2} Dodds P., Dodds T., de Pagter B., Sukochev F. {\it Lipschitz continuity of the absolute value in preduals of semifinite factors.} Integral Equations Operator Theory {\bf 34} (1999), 28--44.
\bibitem{Far6} Farforovskaya Y. {\it An estimate of the nearness of the spectral decompositions of self-adjoint operators in the Kantorovich-Rubinstein metric.}   Vestnik Leningrad. Univ. {\bf 22} (1967), no. 19, 155--156.
\bibitem{Far7} Farforovskaya Y. {\it The connection of the Kantorovich-Rubinstein metric for spectral resolutions of selfadjoint operators with functions of operators.} Vestnik Leningrad. Univ. {\bf 23} (1968), no. 19, 94--97.
\bibitem{Far8} Farforovskaya Y. {\it An example of a Lipschitz function of self-adjoint operators with non-nuclear difference under a nuclear perturbation.} Zap. Nauchn. Sem. Leningrad. Otdel. Mat. Inst. Steklov. (LOMI) {\bf 30} (1972), 146--153.
\bibitem{Grafakos} Grafakos L. {\it Classical Fourier analysis.} Third edition. Graduate Texts in Mathematics, {\bf 249}. Springer, New York, 2014.
\bibitem{HN}  Halmos P., von Neumann J. {\it Operator methods in classical mechanics. II.} Ann. of Math. (2) {\bf 43}, (1942). 332--350.
\bibitem{KS} Kalton N., Sukochev F. {\it Symmetric norms and spaces of operators.} J. Reine Angew. Math. {\bf 621} (2008), 81--121.
\bibitem{Kato} Kato T. {\it Continuity of the map $S \mapsto |S|$ for linear operators.} Proc. Japan Acad. {\bf 49} (1973) 157--160.
\bibitem{Kosaki} Kosaki H. {\it Unitarily invariant norms under which the map $A \mapsto \vert A\vert$ is continuous.} Publ. Res. Inst. Math. Sci. {\bf 28} (1992), 299--313.
\bibitem{Krein} Krein M. {\it Some new studies in the theory of perturbations of self-adjoint operators.} First Math. Summer School, Part I (Russian), Izdat. \lq\lq Naukova Dumka\rq\rq, Kiev, 1964, pp. 103--187.
\bibitem{LSZ} Lord S., Sukochev F., Zanin D. {\it Singular traces. Theory and applications.} De Gruyter Studies in Mathematics, {\bf 46}. De Gruyter, Berlin, 2013.
\bibitem{NazarovPeller} Nazarov F., Peller V. {\it Lipschitz functions of perturbed operators.} C. R. Math. Acad. Sci. Paris {\bf 347} (2009), 857--862.
\bibitem{PSW}  de Pagter B., Witvliet H., Sukochev F. {\it Double operator integrals.} J. Funct. Anal. {\bf 192} (2002), no. 1, 52--111.
\bibitem{Parcet} Parcet J. {\it Pseudo-localization of singular integrals and noncommutative Calder\'on-Zygmund theory.} J. Funct. Anal. {\bf 256} (2009), 509--593.
\bibitem{PellerHankel} Peller V. {\it Hankel operators in the theory of perturbations of unitary and selfadjoint operators.} Funktsional. Anal. i Prilozhen. {\bf 19} (1985), no. 2, 37--51, 96.
\bibitem{PotapotSukochev1} Potapov D., Sukochev F. {\it Lipschitz and commutator estimates in symmetric operator spaces.} J. Operator Theory {\bf 59} (2008) 211--234.
\bibitem{PotapovSukochev2} Potapov D., Sukochev F. {\it Operator-Lipschitz functions in Schatten-von Neumann classes.} Acta Math. {\bf 207} (2011), no. 2, 375--389.
\bibitem{PScrelle} Potapov D., Sukochev F. {\it Unbounded Fredholm modules and double operator integrals.} J. Reine Angew. Math. {\bf 626} (2009), 159--185.
\bibitem{Ran} Randrianantoanina N. {\it A weak type inequality for non-commutative martingales and applications.} Proc. London Math. Soc. (3) {\bf 91} (2005), no. 2, 509--542.
\bibitem{Salle} de la Salle M. {\it A shorter proof of a result by Potapov and Sukochev on Lipschitz functions on $S_p.$} arXiv:0905.1055.
\bibitem{Stein} Stein E. {\it Singular integrals and differentiability properties of functions.} Princeton Mathematical Series, No. {\bf 30} Princeton University Press, Princeton, N.J.
\bibitem{Sbik} Sukochev F. {\it On a conjecture of A. Bikchentaev.} Proc. Sympos. Pure Math., {\bf 87}, Amer. Math. Soc., Providence, R.I., 2013.
\bibitem{Sind} Sukochev F. {\it Completeness of quasi-normed symmetric operator spaces.} Indag. Math. (N.S.) {\bf 25} (2014), no. 2, 376--388.
\end{thebibliography}
\end{document}